\documentclass{article}
\usepackage{amsmath, amsthm, amscd, amsfonts,amssymb,amsgen,amsopn,amsbsy}
\usepackage{eufrak,amscd,bezier,latexsym,mathrsfs,enumerate,multirow}
\usepackage{graphicx}
\usepackage{subfigure}

\usepackage{lineno,hyperref}
\usepackage{mathtools}
\usepackage{arydshln}
\usepackage{float}
\usepackage{caption}
\usepackage{rotating}
\usepackage{tikz}
\usepackage{pdflscape}
\usepackage{fancyhdr,multicol,color,xcolor}
\usepackage[utf8]{inputenc}
\usepackage{float}
\usepackage{soul}
\usepackage{algorithmic}
\usepackage{algorithm}
\usepackage{enumitem}

\modulolinenumbers[5]

\newcommand{\mat}{\left[\begin{array}}
\newcommand{\rix}{\end{array}\right]}

\numberwithin{equation}{section} \numberwithin{equation}{section}
\label{sec.sub.gdn}

\newtheorem{proposition}{Proposition}

\newtheorem{Theorem}{Theorem}[section]

\newtheorem{Corollary}[Theorem]{Corollary}
\newtheorem{lemma}[Theorem]{Lemma}

\newtheorem{Example}[Theorem]{Example}
\newtheorem{Remark}[Theorem]{Remark}
\numberwithin{equation}{section}

\begin{document}
\title{\bf Intra-Basis Multiplication of Polynomials Given in Various Polynomial Bases}
\author{{S. Karami$^1$,~ M. Ahmadnasab$^2$,~M. Hadizadeh$^4$ and~A. Amiraslani$^{3,4\footnote{Corresponding Author: \texttt{amirhosseinamiraslan@capilanou.ca}}}$}
\\
\\
{\em \small   $^1$ Department of Mathematics, Institute for Advanced Studies in Basic Sciences (IASBS), Zanjan, Iran} \\
{\em \small   $^2$ Department of Mathematics, University of Kurdistan, Sanandaj, Iran} \\
{\em \small   $^3$ Faculty of Mathematics, K. N. Toosi University of Technology, Tehran, Iran
 } \\
{\em \small  $^4$ School of STEM, Department of Mathematics, Capilano University, North Vancouver, BC, Canada}}
\maketitle{}


\begin{abstract}
Multiplication of polynomials is among key operations in computer algebra which plays an important role in developing techniques for other commonly used polynomial operations such as division, evaluation/interpolation, and factorization. Despite its success, at least in dealing with orthogonal polynomial bases and the extensive research in that area, using explicit  representations for multiplications of general polynomials in ``degree-graded'' as well as ``non-degree-graded'' polynomial bases without any change of basis has not been exhaustively studied yet.  Even though it is tempting to convert a given polynomial basis to the monomial or one of the orthogonal polynomial bases, it should be noted that  
 a change of basis  may  increase the error propagation not present in the given bases which in turn causes numerical instability. In this paper, we present formulas and techniques for polynomial multiplications expressed in a variety of well-known polynomial bases.  In particular, we take into consideration degree-graded polynomial bases including, but not limited to orthogonal polynomial bases and non-degree-graded polynomial bases including the Lagrange and the Bernstein polynomials bases.
 Our approach is mainly about defining  operational matrices, which intrinsically serve as  basis functions, for determining  explicit matrix-vector representations of intra-basis multiplications of polynomials in a given basis by basis vectors of fixed degrees.  This representation simply  uses either the parameters of three-term recurrence relations (for degree-graded polynomial bases) or lifting relations (for non-degree-graded polynomial bases) in that basis. 
 The proposed framework often leads to well-structured operational matrices. Finally, an application of the presented operational matrices in constructing the stochastic Galerkin matrices is provided.
\\
\\


\end{abstract}


\section{Introduction}


There is an important and ongoing thread of research into polynomial computation using polynomial bases other than the standard monomial power basis. Such problems have interesting applications in several areas including  approximation theory \cite{a, Giorgi2012IEEE, bb, k}, cryptography \cite{app3}, coding \cite{app2} and computer science \cite{app1}.  An important question that has not received enough attention is if it would be possible to have multiplication algorithms  for a variety of polynomials expressed in bases other than the monomial basis.


A key motivation for this interest in polynomial computation using alternatives to the monomial basis is that conversion between bases can be unstable (more commonly so from other bases to the monomial basis) and the instability increases with the degree (see e.g., \cite{hermann1996NA}). The complications arising from such computations motivate explorations of hybrid symbolic-numeric techniques for polynomial computation that are relevant to researchers interested in computer algebra and numerical analysis.

For instance, investigations in~\cite{berrut2004SIAM} and~\cite{higham2004IMA} show that working directly in the Lagrange polynomial basis is both numerically stable and efficient, more stable and efficient than had been heretofore credited widely in the numerical analysis community. These and similar results strengthen the motivation for examining algorithms for direct manipulation of polynomials given in polynomial bases other than the standard monomial power basis.  Figure~\ref{fig:diagarm} summarizes this idea which we intend to use here as well. This work attempts the dashed road. The symbols P, S, ${\rm P_M}$ and ${\rm S_M}$ stand for the ``\underline{P}roblem'' to be solved in an arbitrary basis, the ``\underline{S}olution'' in the same basis, the ``\underline{P}olynomial after conversion to \underline{M}onomial basis'', and the ``\underline{S}olution in the \underline{M}onomial basis'', respectively.
\vspace{-0.4cm}

\begin{figure}[H]
\label{fig:diagarm}
\begin{equation}
\begin{array}{ccc}
  {\rm P} & \dashrightarrow & {\rm S} \\
  \Downarrow &  & \Uparrow \\
  {\rm P_M} & \Rightarrow & {\rm S_M} \nonumber
\end{array}
\end{equation}
\caption{}\label{fig:diagarm}
\end{figure}
\vspace{-0.5cm}

This work is mainly geared toward implementing methods for direct multiplication of polynomials given in polynomial bases. To this end, we look at two important families of polynomial bases:
\begin{enumerate}
  \item degree-graded polynomial bases such as orthogonal polynomial bases and Newton basis, and
  \item non-degree-graded polynomial bases including the Lagrange and Bernstein polynomial bases.
\end{enumerate}

Most of the proposed techniques are based on matrix-vector representations. This enables us to use some known and straightforward techniques from linear algebra which make things a lot easier to follow and, if necessary, extend.

There are  some fast algorithms in literature  for  direct multiplication of polynomials in certain polynomial bases. However, these approaches  are mostly limited to orthogonal polynomial bases, and due to  some stability issues,  it is generally  impossible  to extend them to various non-orthogonal degree-graded polynomial bases as well as non-degree-graded polynomial bases. In the past three decades, some important methods for fast orthogonal  transformations expressed in a variety of matrix-vector structures have been developed (see e.g.. \cite{30, 1, 32, 20, 21, 14, 3, 35, 2}). 

Besides,  writing the product of two polynomials $P_n(x)$ and $P_m(x)$ of degrees $m$ and $n$, respectively, given in an orthogonal polynomial basis as the linear expansion of polynomials in the same basis is a well-known problem called ``the linearization of the product''. The goal is to find the coefficients $g_{mnk} $ such that $P_n(x)P_m(x)=\sum_{k=0}^{n+m}g_{nmk} P_k(x)$.
	Methods and approaches for computing the $g_{nmk}$  have been
	developed so far. For classical families of orthogonal polynomial bases, explicit expressions have been obtained, usually in terms of generalized hypergeometric series, using important  characterizing properties: recurrence relations, generating functions, orthogonality
	weights, etc. (see e.g., \cite{Ronveaux, Andrews1999Cambridge, Markett1994ConstrApprox, Sanchez1999Physics} and references  therein).

We intend to find direct multiplication of polynomials in degree-graded as well as non-degree-graded polynomial bases. In particular and in comparison with~\cite{kleindienst1993chemistry}, all we need to know here about an orthogonal polynomial basis such as but not limited to the Chebyshev polynomial basis (and other degree-graded polynomial bases for that matter) are the coefficients of its ``three-term recurrence relation''.  However, to the best of our knowledge, there has not been any comprehensive study in the literature on direct multiplication of polynomials in various polynomial bases. As such, it is important to implement algorithms to handle direct multiplication on polynomials in such bases. This is for instance of great interest to rigorous computing which aims to guarantee numerical approximation of functions through certified bounds on the results.

 Our main concern in this work is to introduce matrix-vector representations of polynomial multiplications with either specific three-term recurrence relations for degree-graded polynomial bases or lifting relations for non-degree-graded polynomial bases. The suggested ways for representing polynomial multiplications inherit properties, such as computational complexity and memory involved, from ideas that exploit the recurrence or lifting relations between the basic functions of each class of polynomials. As such, no redundant arithmetic operations, besides those of the original idea, exist. Although we do not intend to examine or compare the complexity properties obtained from this form of representation, there are many simple indications that these polynomial multiplication representations not only lend themselves to linear algebra tools to achieve such multiplications, but also the inherent explicit relations make them ready for more effective and innovative implementations.


The structure of this paper is as follows. We start by reviewing some basic properties of  a variety of polynomials including degree-graded and non-degree-graded polynomial bases in Section 2.  Section 3 expresses the main result  of the paper related to the structure of an operational matrix for the  intra-basis multiplication of two arbitrary polynomials with certain degree given in a polynomial basis. The  representation of direct multiplication of  general polynomials is discussed  in Section 4. To this end, the  multiplication of a basis element of a degree-graded polynomial basis by a basis vector of the same polynomial basis with a given degree is represented in a matrix-vector form. Section 5 concerns the direct polynomial multiplication formulas in two common non-degree-graded polynomial bases, i.e. the Bernstein and Lagrange polynomial  bases.  To derive those formulas, we show how to rewrite a polynomial given in a non-degree-graded polynomial basis of certain degree as a polynomial in the same basis with a higher degree using ``lifting'' matrices. Finally, we illustrate in Section 6 an application of the presented techniques which arises in the discretization of linear partial differential equations with random coefficient functions.

\section{Preliminaries and some notations}
We go over some basic definitions and preliminary results related to various polynomial bases in this section.
\subsection{Degree-graded polynomial bases}
Consider a family of real polynomials
$\{\phi_j(x)\}_{j=0}^{\infty}$ with $\phi_j(x)$ of degree $j$
which satisfy a three-term recurrence relation:
\begin{equation} \label{eq.rec}
x \phi_j(x)=\alpha_j \phi_{j+1}(x) +\beta_j \phi_j(x) +
\gamma_j\phi_{j-1}(x), \quad \quad j=0,1,2,\ldots,
\end{equation}
where the $\alpha_j,\;\beta_j,\;\gamma_j$ are
real, $\phi_{-1}(x)=0,~~ \phi_0(x)=1$, and  $  \alpha_j = \frac{k_j}{k_{j+1}} \ne 0 $, with  $k_j$ is the
leading coefficient of $\phi_{j}(x)$.

Generally, any sequence of polynomials $ \{\phi_j(x)\}_{j=0}^\infty $, with $ \phi_j(x) $ of degree $ j $ is {\it degree-graded}, satisfies~\eqref{eq.rec} and obviously forms a linearly independent set, but is not necessarily orthogonal~\cite{amiraslani2016differentiation}. For instance, one can easily observe that the standard basis and Newton basis also satisfy
\eqref{eq.rec} with $ \alpha_j=1,~ \beta_j=0,~ \gamma_j=0 $ and $ \alpha_j=1,~ \beta_j=\tau_j,~ \gamma_j=0 $, respectively,and the $ \tau_j $ are the nodes where the function values are given (see e.g. ~\cite{amiraslani2019differentiation} for more details).

\subsection{Non-degree-graded polynomial bases}
Among other things, in a  {\it non-degree-graded} polynomial basis of degree $n$, i.e., $\{\upsilon_{j, n}(x)\}_{j=0}^{n}$, each basis element $\upsilon_{j, n}(x)$ is itself of degree $n$ and $\sum_{j=0}^{n}{\upsilon}_{j, n}(x)= 1$. It is known that the elements of a non-degree-graded polynomial basis do not form three-term recurrence relations similar to~\eqref{eq.rec}. However, they have other important and interesting properties that enable us to link their basis elements to polynomial bases of different degrees. 

In this work, we particularly take two well-known classes of non-degree-graded polynomial bases, namely the Bernstein and the Lagrange polynomial bases into consideration.  We first  recall some classical definitions taken from ~\cite{farouki1988CAD, berrut2004SIAM}:  
\\

A Bernstein polynomial (also called B\' ezier polynomial) defined over the interval $[a,b]$ has the form
\begin{equation}
{b_{j,n}}(x) = \binom{n}{j} \frac{{{{(x - a)}^j}{{(b - x)}^{n - j}}}}{{{{(b - a)}^n}}},\qquad j=0,\ldots, n.\label{6}
\end{equation}

Note that this is not a typical scaling of the Bernstein polynomials, however this scaling makes matrix notations related to this basis slightly easier to write. Observe that the Bernstein polynomials are nonnegative in $[a,b]$, i.e., $b_{j,n}(x)\ge 0,$ for all $x\in [a,b].$

For a function $f(x) \in C[a,b]$, a polynomial approximation $P(x)$ of degree $n$ written using Bernstein basis is of the form
\begin{equation}\label{7}
P(x) =\sum\limits_{j = 0}^n {{c_j}{b_{j,n}}(x)}= {\bf c}^T{\bf b}_n,
\end{equation}
where ${\bf b}_n(x)=\left[ {{b_{0,n}}(x)},  \cdots,  {{b_{n,n}}(x)}
\right]^T$,  ${\bf c}= [c_0, \cdots, c_n]^T$  
with $c_j=f( a+\frac{(b-a)j}{n} ).$
\\

The following lemma from~\cite{buse2008CAGD} describes an important property of the Bernstein polynomial basis:
\begin{lemma}\label{bern1lem}
	If $ b_{j,n}(x) $ is the j-th Bernstein polynomial basis of a degree $n$, then
	\begin{equation*}\label{bern2}
	b_{j,n}(x)=\left( \frac{j+1}{n+1}\right) b_{j+1,n+1}(x)+\left(\frac{n+1-j}{n+1}\right) b_{j,n+1}(x), \qquad  j=0, \dots, n.
	\end{equation*}
\end{lemma}
\vspace{0.3cm}

The next important non-degree-graded of interest is the Lagrange polynomial basis:

Suppose that a function $f(x)$ is sampled at
$n+1$ distinct points $\tau_0,\,\tau_1,\ldots,\tau_n$, and write
$p_j:=f(\tau_j)$. This may also be shown as $\{(\tau_j, p_j)\}_{j=0}^{n}$. The Lagrange basis polynomials of degree $n$ are then defined by
\begin{equation}
\label{eq.lag} L_{j, n}(x)=\frac{\ell_n(x)w_{n,j}}{x-\tau_j},\quad \quad
j=0,1,\ldots,n
\end{equation}
where the ``weights'' $w_{n,j}$ are  $ w_{n,j}=\prod_{m=0,\,m\ne j}^n\frac{1}{\tau_j-\tau_m},$
and $ \ell_n(x)=\prod_{m=0}^n(x-\tau_m).$

Then $P(x)$ can be
expressed in terms of its samples in the following form
\begin{equation}\label{eq:polyL}
P(x)= {\bf p}^T{\bf L}_n(x),
\end{equation}
where ${\bf p}= \left[ p_0,  \cdots, p_n\right]^T$ and ${\bf L}_n(x)= \left[L_{0, n}(x),  \cdots, L_{n, n}(x)\right]^T$.
\\
The Lagrange polynomial basis has  also  the following  main  property which is due in \cite{berrut2004SIAM}:
\begin{lemma}{\rm (From \cite{berrut2004SIAM})}\label{lag1lem}
	If we eliminate a node, say $\tau_k$, of a Lagrange polynomial basis of degree $n$ with nodes $\tau_i$, we have
	$$ 	L_{i,n-1}(x)=\frac{{\ell}_{n-1}(x)w_{n-1,i}}{x-\tau_i},\qquad i=0,1,\dots,n;\quad i \neq k $$
	where $ \ell_{n-1}(x)=\frac{\ell_n(x)}{x-\tau_k} $ and $ w_{n-1,i}= w_{n, i}(\tau_i-\tau_k) $, then
	\begin{equation*}\label{lag1}
	L_{i,n-1}(x)=-\frac{w_{n,i}}{w_{n,k}}L_{k,n}(x)+ L_{i,n}(x).
	\end{equation*}
\end{lemma}
\subsection{Important notations}\label{important.note}
We list below some notations used throughout the paper:
\begin{itemize}
 \item ~~  $\phi_{j, n}(x)$:  The $j$-th (i.e., $j$-th degree) polynomial of a degree-graded polynomial basis of degree $n$.

 \item ~~  $\alpha_j,\;\beta_j,\;\gamma_j$: The real coefficients of  three-term recurrence relation of a degree-graded polynomial.



 \item ~ ${\bf \Theta}_n (x)$: The vector form of a general  polynomial basis of degree $n$  defined as ${\bf \Theta}_n(x)=[\theta_{0,n}(x), \cdots, \theta_{n,n}(x)]^T$, where $\theta_{i,n}(x)$ is $i$-th basis element for $i=0, 1, \cdots, n$.

 \item ~~ ${ \bf \Phi}_n (x)$:  The vector form of a degree-graded polynomial basis of degree $n$ defined as ${\bf\Phi}_{n}(x)=  [\phi_{0}(x), \cdots, \phi_{n}(x)]^T$.



 \item ~~ ${\bf b}_{n}(x)$: The  Bernstein polynomial basis of degree $n$ defined as ${\bf b}_{n}(x)= \mat{cccc}{\bf b}_{0, n}(x)& \cdots &{\bf b}_{n, n}(x)\rix^T$, where ${\bf b}_{i, n}(x)$ for $i= 1, 2, \cdots, n$ is the $i$-th basis element of the Bernstein polynomial bases of degree  $n$ defined on the interval $[a, b]$.



 \item~~   ${\bf L}_{n}(x)$: The Lagrange polynomial basis of degree $n$ defined as  ${\bf L}_{n}(x)= \mat{cccc}{\bf L}_{0, n}(x)& \cdots &{\bf L}_{n, n}(x)\rix^T$, where ${\bf L}_{i, n}(x)$ for $i= 1, 2, \cdots, n$ is the $i$-th basis element of the Lagrange polynomial bases of degree  $n$.

\item ~~$\Psi(x)$ and $ \Xi(x)$: The  matrix-vector representation of  arbitrary polynomials  of degrees $n$ and $m$, respectively,  in terms of  the basis vector ${\bf \Theta}(x)$ defined as $~~\Psi(x)= {\bf \psi}^{(m)} {\bf \Theta}_m(x)$ and $ \Xi(x)= {\bf \xi}^{(n)} {\bf \Theta}_n(x)$, where ${\bf \psi}^{(m)}= \mat{ccccc}\psi_0& \cdots &\psi_m\rix$, and ${\bf \xi}^{(n)}= \mat{ccccc}\xi_0&\cdots &\xi_n\rix $.

\item  ~~ ${\bf \tilde{H}}_{n,k}$: The operational matrix of the size $(n+ 1)\times (n+m+1)$ of the multiplication of a polynomial of degree $0\leq k\leq m$ in a given basis by its basis vector of a given degree $n$.

 \item ~~ ${\bf \mathcal{H}}_{n,m}$: The operational matrix of the size $(n+ 1)\times (n+m+1)$  of the multiplication of a polynomial of degree $n$ in a given polynomial basis by a polynomial of degree $m$ in the same polynomial basis.

\item~~ ${\bf T}_{n, m}$: The  lifting matrix between the Bernstein polynomial basis of degree $n$ and the Bernstein polynomial basis of degree $m (> n)$.

\item~~ ${\bf R}_{n, m}$: The  lifting matrix of the Lagrange polynomial basis of degree $n$ and the Lagrange polynomial basis of degree $m (> n)$.

 \item ~~$\otimes$:  Kronecker product of matrices.

\end{itemize}
\section{Main Result}\label{mainthm}
In this section, we state our main result which provides an insight into  the rest of this paper.
In order to provide our main result about intra-basis multiplication of  polynomials, let us  assume $\Psi(x)$ and $\Xi(x)$  be two arbitrary polynomials  of degrees $m$ and $n$, respectively, given in a polynomial basis as
	 \begin{equation}\label{eq.PsiXi.polys}
	 \begin{aligned}
	 \Psi(x)&= {\bf \psi}^{(m)} {\bf \Theta}_m(x),\\
	 \Xi(x)&= {\bf \xi}^{(n)} {\bf \Theta}_n(x),
	 \end{aligned}
	 \end{equation}
	 for ${\bf \psi}^{(m)}= \left[\psi_0 ~~ \psi_1~~ \cdots ~~\psi_m  \right]$,~~ ${\bf \xi}^{(n)}= \left[\xi_0~~\xi_1~~ \cdots ~~\xi_n\right]$,   ${\bf \Theta}_m(x)=[\theta_{0,m}(x), \cdots, \theta_{m,m}(x)]^T$, and ${\bf \Theta}_n(x)=[\theta_{0,n}(x), \cdots, \theta_{n,n}(x)]^T$, where $\theta_{i,n}(x)$ denotes the $i$-th basis element in the polynomial basis.

Using the introduced notations in Section~\ref{important.note}, we give the following key lemma:
\begin{lemma}\label{lemOperationalMatrix}
 For any non-negative  integers $n, m$ and $k$, where   $0 \leq k \leq m$,
 there exists a unique operational matrix ${\bf \tilde{H}}_{n,k}$ of size $(n+1) \times (n+m+1)$  such that:
\begin{equation}\label{eqq1}
\theta_{k,m}(x){\bf \Theta}_n(x)={\bf \tilde{H}}_{n,k}{\bf \Theta}_{n+m}(x).\end{equation}
	\end{lemma}
\begin{proof}
The proof is fairly straightforward. Note that theoretically, in any polynomial basis, we can always expand and then uniquely write the entries of $\theta_{k,m}(x){\bf \Theta}_n(x)$ as well as the entries of ${\bf \Theta}_{n+m}(x)$ in terms of the standard basis elements. That way, we can find the operational matrix in the standard basis.

Moreover, it can be easily verified that there is always an invertible transformation matrix between any given polynomial basis and the standard basis. That gives us the desired operational matrix in the given basis.
\end{proof}

This lemma shows the crucial role that the operational matrix  ${\bf \tilde{H}}_{n,k}$ plays in performing intra-basis multiplications of polynomials for a wide variety of polynomial  bases.  Consequently, we can state the following important theorem:
%
\begin{Theorem}\label{th1}
	Let 	$~\Xi(x)$ and $~\Psi(x)$ be two arbitrary polynomials  of degrees $n$ and $m$, respectively as given in \eqref{eq.PsiXi.polys}. Then we can find their multiplication through	
	\begin{equation}\label{eqqt1b}
	\Xi(x)\Psi(x)= {\bf \xi}^{(n)} \mathcal{H}_{n,m}{\bf\Theta}_{n+m}(x),
	\end{equation}
where
	\begin{equation}\label{eqqt1}
	{\bf \mathcal{H}}_{n,m}= \sum_{k= 0}^{m}{\psi_k{\bf \tilde{H}}_{n, k}}.
	\end{equation}
\end{Theorem}
\begin{proof}
	We have
	\begin{equation}
	\begin{aligned}
	\Xi(x)\Psi(x)&={\bf \xi}^{(n)} {\bf \Theta}_n(x) \sum_{k=0}^{m}\psi_k\theta_{k,m}(x)\\
	&= {\bf \xi}^{(n)} \sum_{k=0}^{m}\psi_k\theta_{k,m}(x) {\bf \Theta}_n(x)\\
	&= {\bf \xi}^{(n)} \sum_{k=0}^{m}\psi_k{\bf \tilde{H}}_{n,k}{\bf \Theta}_{n+m}(x)\\
	&= {\bf \xi}^{(n)} {\bf \mathcal{H}}_{n,m}{\bf \Theta}_{n+m}(x)
	\end{aligned}
	\end{equation}
	
\end{proof}

This theorem indicates that intra-basis multiplications of polynomials depend on the structure of the operational matrices ${\tilde{{\bf H}}}_{n,k}$.  In the forthcoming sections,  we focus on the structures of these operational matrices for a general type of degree-graded polynomial  bases including orthogonal bases as well as two commonly used non-degree-graded polynomial bases, namely  the Lagrange and Bernstein polynomial bases.

\subsection{Powers of a Polynomial}

An interesting result of Theorem~\ref{th1} is a formula for integer powers of an arbitrary  polynomial in a given polynomial basis which can be  expressed as follows:
\begin{Corollary}
	If $\Xi(x)= {\bf \xi}^{(n)} {\bf \Theta}_n(x)$, then  for a positive integer $ p (> 1)$:
	\begin{equation}\label{gardedn}
	\displaystyle{\Xi}^p(x)= {\bf \xi}^{(n)}(\prod_{j=1}^{p-1}{\bf \mathcal{H}}_{j\times n,n}){\bf \Theta}_{p\times n}(x),
	\end{equation}
such that  $\displaystyle{{\bf \mathcal{H}}_{j\times n,n}= \sum_{k= 0}^{n}\xi_i {\bf \tilde{H}}_{j\times n, k}}, ~(j= 1,  \cdots, p-1)$, where ${\bf \tilde{H}}_{j\times n, k}$ is given in Lemma \ref{lemOperationalMatrix}.
\end{Corollary}
\begin{proof}
	The proof is by induction in $p$. When $p = 2$, we use the statement of Theorem~\ref{th1} to get

	\[
	\Xi^2(x) = {\bf \xi}^{(n)} \displaystyle{({\bf \mathcal{H}}_{n, n}) \Phi_{2n} (x)},
	\]

and similarly

\[
	\Xi^3(x) = \Xi(x)\Xi^2(x)= {\bf \xi}^{(n)} \Phi_{n} (x) {\bf \xi}^{(n)} \displaystyle{({\bf \mathcal{H}}_{n, n}) \Phi_{2n} (x)},
	\]
which can be written using Theorem~\ref{th1} as

\[
	\Xi^3(x) = {\bf \xi}^{(n)} \displaystyle{({\bf \mathcal{H}}_{n, n}{\bf \mathcal{H}}_{2\times n, n}) \Phi_{3n} (x)}.
	\]

	Now, the inductive hypothesis for $\displaystyle{\Xi^{p - 1}(x)}$ becomes
	\[
	\displaystyle{\Xi^{p}(x)}= \displaystyle{\Xi (x)} ~\displaystyle{\Xi^{p - 1}(x)} = {\bf \xi}^{(n)} {\bf \Phi}_n(x)  {\bf \xi}^{(n)} (\prod_{j=1}^{p -2}{\bf \mathcal{H}}_{(j\times n),n}){\bf \Phi}_{((p - 1) \times n)}(x).
	\]

	Using Theorem~\ref{th1} and the notation of $\displaystyle{{\bf \mathcal{H}}_{(j\times n),n}}$ we arrive at
	\[
	\displaystyle{\Xi^{p}(x) = {\bf \xi}^{(n)} (\prod_{j=1}^{p-1}{\bf \mathcal{H}}_{(j\times n),n}){\bf \Phi}_{((p - 1) \times n + n)}(x)} = {\bf \xi}^{(n)} (\prod_{j=1}^{p-1}{\bf \mathcal{H}}_{(j\times n),n}){\bf \Phi}_{(p \times n)}(x).
	\]
\end{proof}

\section{Direct Multiplication of Polynomials in Degree-Graded Bases}\label{sec.Degree-Graded}
In this section, our aim is to determine the structure of the operational matrix ${\tilde{{\bf H}}}_{n,k}$  
 for a fixed non-negative integer $k$,  $0\leq k\leq m$, in the case of degree-graded polynomial bases.  We also show that how the operational matrix ${\tilde{{\bf H}}}_{n-1,k}$ is embedded  in   ${\tilde{{\bf H}}}_{n,k}$.

Note that to make Theorem~\ref{th1} consistent for all polynomial bases, 
 we use a slightly different notation for degree-graded polynomial bases in this section than we use elsewhere in this work. 
 In particular, we denote the family of degree-graded polynomial basis of a maximum degree of $n$ by $\{\phi_{j, n}(x)\}_{j=0}^{n}$, where $\phi_{j, n}(x)$ is of degree $j$. The second index $n$ may seem redundant in $\phi_{j, n}(x)$, but it is important in what follows 
 to maintain the consistency and universality of the important results of Section \ref{mainthm} across all polynomial bases.
 Under this notation,  in any degree-graded basis regardless of the maximum degree, the basis elements of the same degree are equal, i.e., $$\phi_{j, m}(x)= \phi_{j, n}(x);~~~~0\leq j\leq\min(m, n).$$

\subsection{Multiplication of Basis Elements}\label{MBED}
The following important lemma defines a matrix  operator that is used in determining the result of the multiplication of a basis element  of a degree-graded polynomial of a certain degree by the basis vector of a given degree. This can be viewed as a special case of Lemma~\ref{lemOperationalMatrix} for the degree-graded polynomial bases.

\begin{lemma}\label{lem1}
	For every non-negative integers $k$ and $n$, we have
	\begin{equation}\label{eq9}
		\phi_{k, m}(x){\bf \Phi}_n(x)={\bf \tilde{H}}_{n, k}{\bf \Phi}_{n+m}(x),\end{equation}
	such that  
  \begin{equation}\label{Htilde}
	{\bf \tilde{H}}_{n, k}=\left[ \begin{array}{c; {2pt/2pt}c c}
	{\bf H}_{n, k}& {\bf 0}_{(n+1)\times (m-k)}
	\end{array} \right],\qquad \quad  k= 0, 1, \cdots, m,
	\end{equation}
	where ${\bf H}_{n,k}$ is a matrix of the size $(n+1)\times(n+k+1)$  with the following entries:
	\begin{small}
		\begin{equation}\label{Hmatrix}
			{\bf H}_{n,k}[i,j]=\begin{cases}
				1, ~~~~~~~~~~~~~~~~~~~~~~~~~~~~~~~~~~~i=1,~ j=k+1,\\
				0, ~~~~~~~~~~~~~~~~~~~~~~~~~~~~~~~~~~~ j \leq |k+1-i|  $ {\rm or} $ j\geq k+1+i,\\
				\frac{1}{\alpha_{i-2}}(	\alpha_{j-2}{\bf H}_{n,k}[i-1,j-1]+(\beta_{j-1}-\beta_{i-2}){\bf H}_{n,k}[i-1,j]+\gamma_{j}{\bf H}_{n,k}[i-1,j+1]-\gamma_{i-2}{\bf H}_{n,k}[i-2,j]), ~~{\rm  elsewhere.}
			\end{cases}
		\end{equation}
	\end{small}

	The entries ${\bf H}_{n,k}[i,j]$ are set to $0$, for any negative or zero index or when  $j >n+k+1$.
\end{lemma}
\begin{proof}
	For a basis element $\phi_{k, m}(x)$ of a fixed degree
	$k$ of a degree-graded polynomial basis of the maximum degree $m$, we have
	\begin{equation}\label{eq1.5}
	\phi_{0, n}(x)\phi_{k, m}(x)=\phi_{k, n+m}(x).
	\end{equation}

This can be written as $$\phi_{0, n}(x)\phi_{k, m}(x)={\tilde {\bf h}}^T_{0, k} {\bf {\Phi}}_{n+m}(x),$$ where

\begin{equation*}
			{ \tilde{\bf h}}_{0,k}[i]=\begin{cases}
				1, ~~i=k+1,\\
				0, ~{\rm  elsewhere.}
			\end{cases}
\end{equation*}

	Multiplying \eqref{eq1.5} by $x$, we get
	\begin{equation*}\label{eq2.5}
	x\phi_{0, n}(x)\phi_{k, m}(x)=x\phi_{k, n+m}(x).
	\end{equation*}
	
	From the recurrence relation on each side, 
	 we have
	\begin{equation*}\label{eq3.5}
	\Big(\alpha_{_{0}}\phi_{1, n}(x)+\beta_{_{0}}\phi_{0, n}(x)\Big)\phi_{k, m}(x)=\alpha_{_{k}}\phi_{k+1, n+m}(x)+\beta_{_{k}}\phi_{k, n+m}(x)+\gamma_{_{k}}\phi_{k-1, n+m}(x),
	\end{equation*}
	which  gives
	\begin{equation}\label{eq4.5}
	\qquad \phi_{1, n}(x)\phi_{k, m}(x)=\frac{1}{\alpha_{_{0}}}\Big(\alpha_{_{k}}\phi_{k+1, n+m}(x)+(\beta_{_{k}}-\beta_{_{0}})\phi_{k, n+m}(x)+\gamma_{_{k}}\phi_{k-1, n+m}(x)\Big).
	\end{equation}

This can also be written as $$\phi_{1, n}(x)\phi_{k, m}(x)={\tilde{\bf h}}^T_{1, k}{{\bf \Phi}}_{n+m}(x),$$ where

\begin{equation*}
			{\bf \tilde{h}}_{1,k}[i]=\begin{cases}
				\frac{\gamma_k}{\alpha_0}, ~~~~~~~i=k,\\
                \frac{(\beta_k-\beta_0)}{\alpha_0}, ~i=k+1,\\
                \frac{\alpha_k}{\alpha_0}, ~~~~~~~i=k+2,\\
				0, ~~~~~~~~~{\rm  elsewhere.}
			\end{cases}
\end{equation*}

	Similarly, multiplying~\eqref{eq4.5} by $x$ yields
	\begin{equation}\label{5.5}
	\begin{aligned}
	\phi_{2, n}(x)\phi_{k, m}(x)=\frac{1}{\alpha_{_{0}}\alpha_{_{1}}}&\Big(
	\alpha_{_{k+1}}\alpha_{_{k}}\phi_{{k+2}, n+m}(x)+\alpha_{_{k}}(\beta_{_{k+1}}+\beta_{_{k}}-\beta_{_{1}}-\beta_{_{0}})\phi_{{k+1}, n+m}(x)\\
	&+((\beta_{_{k}}-\beta_{_{0}})(\beta_{_{k}}-\beta_{_{1}})-\gamma_{_{1}}\alpha_{_{0}}+\gamma_{_{k}}\alpha_{_{k-1}}+\gamma_{_{k+1}}\alpha_{_{k}})\phi_{k, n+m}(x)\\
	&+\gamma_{_{k}}(\beta_{_{k}}+\beta_{_{k-1}}-\beta_{_{1}}-\beta_{_{0}})\phi_{k-1, n+m}(x)+\gamma_{_{k}}\gamma_{_{k-1}}\phi_{{k-2}, n+m}(x)\Big),
	\end{aligned}	\end{equation}
which, like the previous steps, can be written as $$\phi_{2, n}(x)\phi_{k, m}(x)={\tilde{\bf h}}^T_{2, k}{{\bf \Phi}}_{n+m}(x).$$
	
	Continuing this process, we can obtain $\phi_{3, n}(x)\phi_{k, m}(x),...,\phi_{n, n}(x)\phi_{k, m}(x)$ and write the results in a matrix-vector form.
Then~\eqref{eq9} is obtained, where $${\bf \tilde{H}}_{n, k}= \left[
    \begin{array}{ccc}
        {\bf {\tilde{h}}}^T_{0, k} \\ \hdashline[2pt/2pt]
        {\bf {\tilde{h}}}^T_{1, k} \\ \hdashline[2pt/2pt]
        \vdots \\ \hdashline[2pt/2pt]
        {\bf \tilde{h}}^T_{n, k}
    \end{array}
\right]$$ is given by~\eqref{Htilde}.
\end{proof}

At this point and once the operational matrices ${\bf \tilde{H}}_{n, k}$ are determined using the parameters of the three-term recurrence relation, we can compute direct multiplication of two degree-graded polynomials in a given basis by a basis vector of a given degree by  Theorem~\ref{th1}.

\subsection{More on the Structure of ${\bf H}_{n,k}$}
\label{SA}

In this subsection, some important features and advantages of the operational matrices ${\bf H}_{n,k}$  related to multiplication of degree-graded  polynomials basis are expressed. We  show  that there exist an embedding relation between matrices ${\bf H}_{n,k}$ and ${\bf H}_{n-1,k}$.
In order to gain some intuition,  Figure \ref{Fig-Hkn} displays the sparsity layout (non-zero structure) of ${\bf H}_{n,k}$ given by~\eqref{Hmatrix}  for any  $k \leq n$ and $k \geq n$, respectively. Besides, Table \ref{Tablembym} indicates the number of  nonzero entries of  ${\bf H}_{n,k}$ that need to be independently computed.


\begin{figure}[!h]
	\centering
	\subfigure[ $k \leq n$: The only nonzero entry of the first column appears in the $(k+1)$-st row.]{\includegraphics[width=.465\textwidth]{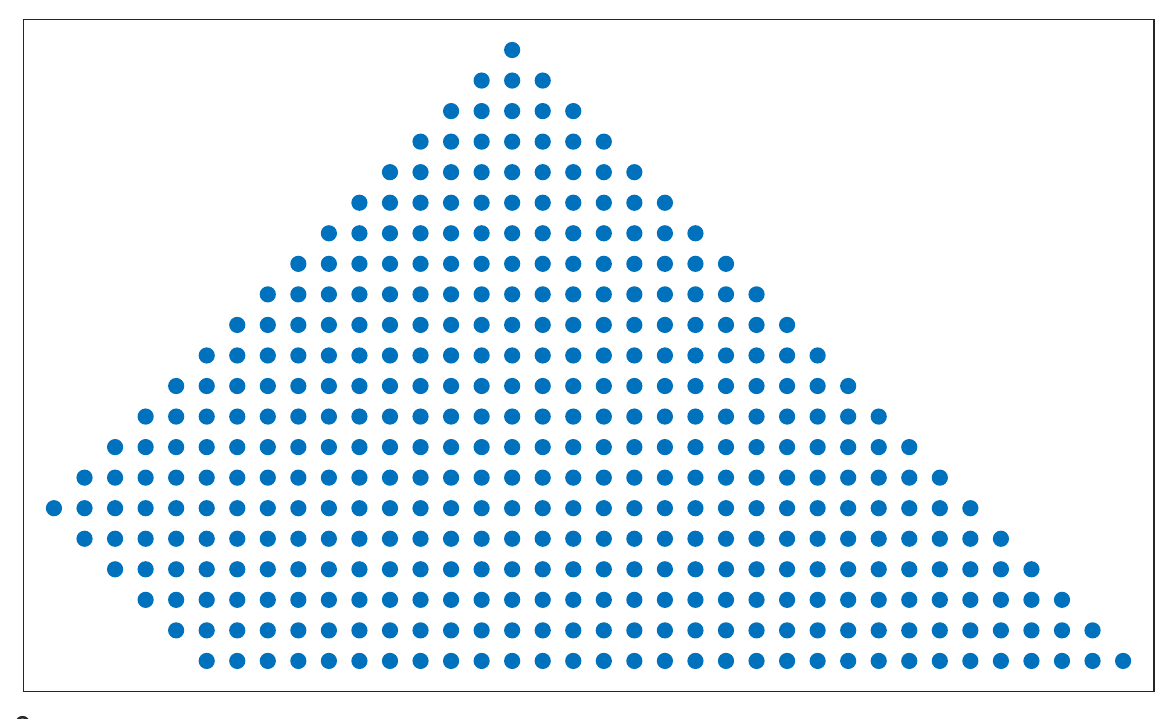} \label{Fig-Hkn1}}\hspace{.25cm}
	\subfigure[$k \geq n$: The first $(k-n)$ columns of the matrix are zero vectors.]{ \includegraphics[width=.512\textwidth]{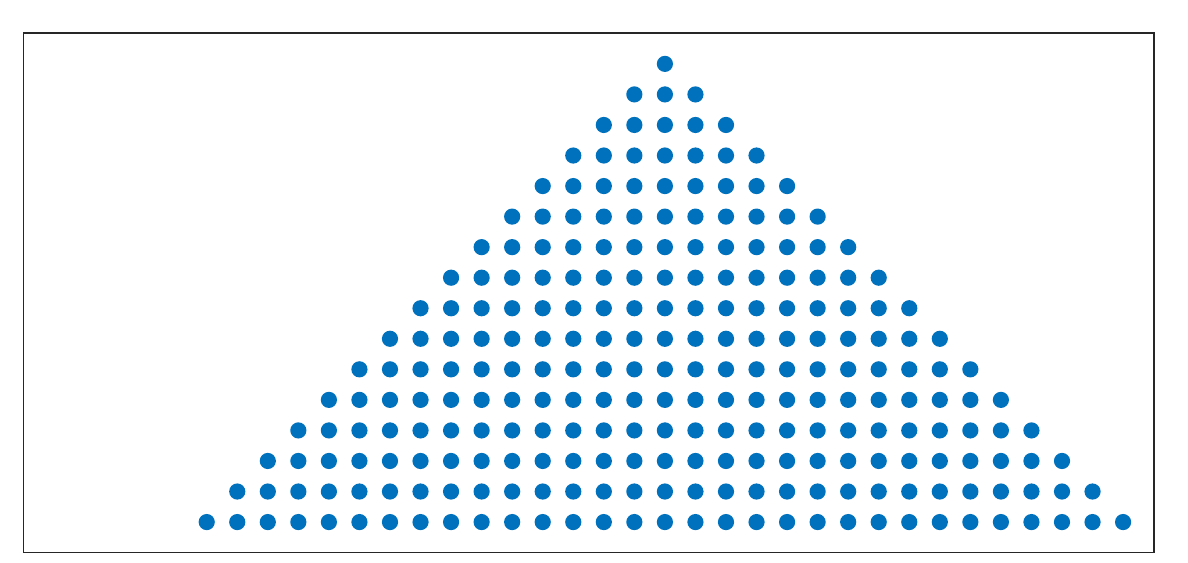}\label{Fig-Hkn2}}
	\caption{\small Non-zero structure of the matrices ${\bf {H}}_{n,k}$ for a general degree-graded polynomial  bases.}
	\label{Fig-Hkn}
\end{figure}

\begin{proposition}
	\label{p1}
	Let $n$ and $k$ be two non-negative integers. Then
	\begin{itemize}
		\item[(a)] the structure of ${\bf H}_{0,k}$ is as follows
		
		\begin{equation*}\label{eqht}
		{{\bf H}_{0, k}}=\left[ \begin{array}{c; {2pt/2pt}c c}
		{\bf 0}_{1\times k}& {\bf 1 }
		\end{array} \right] ~~~\text{and} ~~~ {\bf H}_{0,0} = \left [  \begin{array}{c} {\bf 1} \end{array} \right]
		\end{equation*}
		and  ${\bf H}_{k,0}$ is  the $(k+1) \times (k+1)$ identity matrix ${\bf I}_{k+1}$;
		\item[(b)]  the  matrix $~{\bf H}_{n-1,k}$  is a submatrix of  $~{\bf H}_{n,k}$, in the form of
		\begin{equation}\label{nestedHkn}
		{\bf H}_{n,k} = \left[ \begin{array}{c; {2pt/2pt}cc}
		\begin{array}{c} {\bf H}_{n-1,k}\\ \hdashline[2pt/2pt] {\bf H}_{n,k} [n+1,1:n+k] \end{array}& \begin{array}{c} {\bf 0}_{n\times 1}\\ \hdashline[2pt/2pt] {\bf H}_{n,k} [n+1,n+k+1] \end{array}
		\end{array} \right]
		\end{equation}
		using MATLAB colon notation.
		The last row  is the only part of the matrix
		${\bf H}_{n,k}$ that must be computed and has at most $2k+1$ nonzero entries.
	 The operational  cost for computing the sequence $\{ {\bf H}_{j,k} \}_{j=1}^{n}$ is the same as the one for ${\bf H}_{n,k}$;
	\item[(c)] the last rows of  ${\bf H}_{i,j}$ and ${\bf H}_{j,i}$, for $i \not = j$, are the same.  There is no need to separately calculate the nonzero entries of ${\bf H}_{i, j}$, when $j > i$, which leads to operational savings in computing the matrices ${\bf H}_{i,j}$.	

%
		
		\end{itemize}
\end{proposition}
\begin{table}
	\begin{center} \footnotesize
		{
			\begin{tabular}{|l|ccccccc|}
				\hline
				&  $\phi_{0,n}$    & $\phi_{1,n}$    & $\phi_{2,n}$    &  $\phi_{3,n}$    &  $\cdots$ &  &      $\phi_{n,n}$ \\
				\hline
				${\bf \Phi}_0$   &  ${\bf H}_{0,0} ({\color{red} 0})$&  ${\bf H}_{0,1} ({\color{red} 0})$& ${\bf H}_{0,2} ({\color{red} 0})$& ${\bf H}_{0,3} ({\color{red} 0})$& $\dots$ &  & ${\bf H}_{0,n} ({\color{red} 0})$ \\
				${\bf \Phi}_1$   & ${\bf H}_{1,0} ({\color{red} 0})$&  ${\bf H}_{1,1} ({\color{red} 3})$& ${\bf H}_{1,2} ({\color{red} 0})$& ${\bf H}_{1,3} ({\color{red} 0})$&$\dots$ &  &${\bf H}_{1,n} ({\color{red} 0})$ \\
				${\bf \Phi}_2$  & ${\bf H}_{2,0} ({\color{red} 0})$&  ${\bf H}_{2,1} ({\color{red} 3})$& ${\bf H}_{2,2} ({\color{red} 5})$& ${\bf H}_{2,3} ({\color{red} 0})$&$\dots$ &  &    \\
				${\bf \Phi}_3$   & ${\bf H}_{3,0} ({\color{red} 0})$&  ${\bf H}_{3,1} ({\color{red} 3})$& ${\bf H}_{3,2} ({\color{red} 5})$& ${\bf H}_{3,3} ({\color{red} 7})$& $\ddots$ & & $\dots$  \\
				$\vdots$ & $\vdots$ & $\vdots$ &&& $\ddots$ & $\ddots$ & ${\bf H}_{n-2,n}({\color{red} 0})$ \\
				${\bf \Phi}_{n-1}$   & ${\bf H}_{n-1,0} ({\color{red} 0})$&  ${\bf H}_{n-1,1} ({\color{red} 3})$& ${\bf H}_{n-1,2} ({\color{red} 5})$& ${\bf H}_{n-1,3} ({\color{red} 7})$& $\ddots$ &  $\ddots$ & ${\bf H}_{n-1,n} ({\color{red} 0}) $  \\
				${\bf \Phi}_{n}$   & ${\bf H}_{n,0} ({\color{red} 0})$&  ${\bf H}_{n,1} ({\color{red} 3})$& ${\bf H}_{n,2} ({\color{red} 5})$& ${\bf H}_{n,3} ({\color{red} 7})$& $\dots$ &  $\dots$ & ${\bf H}_{n,n} ({\color{red} 2n+1}) $ \\
				\hline
			\end{tabular}
		}
	\end{center}
	\caption{\label{Tablembym}
		{\footnotesize The number of nonzero entries of each ${\bf H}_{i,j},~~ (i, j=0, \dots, n)$ that need to be independently computed is shown inside the parentheses in front of each matrix.}}
\end{table}

The existing embedding relations as well as the pattern in the number of nonzero entries of each ${\bf H}_{n,k}$ that must be independently generated, enable us to save in computational costs.
In Table \ref{Tablembym}, we determine the number of   new entries for construction  each matrix ${\bf H}_{n,k}$. For $n \geq k$, only $2k+1$ new entries  should be computed. Now, consider  the case $n < k$. Through parts (b), (c) of  Proposition \ref{p1}, the matrix ${\bf H}_{n,k}$ can be obtain by using the matrix ${\bf H}_{n-1,k}$ and the last row of the matrix ${\bf H}_{k,n}$.

Since through part (a) of Proposition \ref{p1}, the matrix ${\bf H}_{0,k}$ does not need any computation, by induction, the matrix ${\bf H}_{n,k}$ does not need any calculation. Actually, the matrix ${\bf H}_{n,k}$ does not need new calculation provided that the matrices ${\bf H}_{k,j}~j=0,\ldots,n$ have already been computed.
For example,  there is no need for new computations for the matrix ${\bf H}_{2,4}$ if the matrices ${\bf H}_{4,0}, {\bf H}_{4,1}$ and ${\bf H}_{4,2}$ are already given. Therefore, for the matrices that are located on the upper side of the diagonal in Table \ref{Tablembym}, we need no new computations provided that the matrices that are located on the lower side of the diagonal have already been given.

For further clarification, the following  example helps better  illustrates the structure of ${\bf H}_{n,k}$ and its applications in the direct multiplication of polynomials in degree-graded  polynomials bases.
\begin{Example}\label{ExCheli}
	\rm Let us consider the alternative orthogonal Legendre polynomial basis also known as the Chelyshkov polynomial basis of second  kind  (see e.g., \cite{p} ). It is a degree-graded polynomial basis on the interval [0, 1] with weight function $w(x) = x$, where it can be expressed in terms of Jacobi polynomials. These  polynomials are also related to the hypergeometric functions and orthogonal exponential polynomials. In recent years, Chelyshkov polynomials have found applications in various fields of approximation theory and numerical analysis, see for example \cite{q, r}.  According to \cite{s}, the recurrence coefficients for this polynomials are $\alpha_i =\frac{i + 2}{4i + 6},~ \beta_i =\frac{2(i+1)^2}{(2i + 3)(2i + 1)}$ and $\gamma_i = - \frac{i }{4 i + 2}$, for $i = 0 , ... , n.$
	
	To have a better understanding of the above results, we construct the sample matrices ${\bf H}_{n,k},~(n,k=0, 1, 2)$ for this polynomial basis. First, note that ${\bf H}_{0,0}={\bf I}_{1 \times 1}$, ${\bf H}_{1,  0}={\bf I}_{2 \times 2}$,  ${\bf H}_{2,  0}={\bf I}_{3 \times 3}$ and
	\begin{equation*}
	\begin{aligned}
	&~{\bf H}_{0,1}=\begin{bmatrix} 0  & 1 \end{bmatrix}, ~~~~~~~~~~~~~~~~~~~~~~~~~~~~~~~~~~~~~~~~~~~~~~~~~~~
	{\bf H}_{0,2}=\begin{bmatrix} 0  & 0 &1 \end{bmatrix},
	\\
	\\
	&~{\bf H}_{1,1}=\begin{bmatrix}
	0  & 1 & 0\\ \frac{1}{2}& \frac{2}{5} & \frac{9}{10}
	\end{bmatrix}
	=\left[ \begin{array}{c; {1.5pt/1.5pt}c}
	\begin{array}{c} {\bf H}_{0,1}\\ \hdashline[1.5pt/1.5pt] \begin{array}{c c}
	\frac{1}{2} & \frac{2}{5}
	\end{array} \end{array}& \begin{array}{c} {\bf 0}_{1\times 1}\\ \hdashline[1.5pt/1.5pt] \frac{9}{10} \end{array}
	\end{array} \right],
	~~~~~~
	{\bf H}_{1,2}=\begin{bmatrix}
	0 &0  &1 & 0\\ 0 & \frac{3}{5} & \frac{16}{35} & \frac{6}{7}
	\end{bmatrix}
	=	\left[ \begin{array}{c; {2pt/2pt}cc}
	\begin{array}{c} {\bf H}_{0,2}\\ \hdashline[2pt/2pt] \begin{array}{c c c}
	0 & \frac{3}{5} & \frac{16}{35}
	\end{array} \end{array}& \begin{array}{c} {\bf 0}_{1\times 1}\\ \hdashline[2pt/2pt] \frac{6}{7} \end{array}
	\end{array} \right],
	\\
	\\
	&~{\bf H}_{2,1}=
	\left[ \begin{array}{c; {2pt/2pt}cc}
	\begin{array}{c} {\bf H}_{1,1}\\ \hdashline[2pt/2pt] \begin{array}{c c c}
	0 & \frac{3}{5} & \frac{16}{35}
	\end{array} \end{array}& \begin{array}{c} {\bf 0}_{2\times 1}\\ \hdashline[2pt/2pt] \frac{6}{7} \end{array}
	\end{array} \right],
	~~~~~~~~~~~~~~~~~~~~~
	{\bf H}_{2,2}=
	\left[ \begin{array}{c; {2pt/2pt} c}
	\begin{array}{c} {\bf H}_{1,2}\\ \hdashline[2pt/2pt] \begin{array}{c c c c}
	\frac{1}{3} & \frac{32}{105} & \frac{24}{35}& \frac{32}{63}
	\end{array} \end{array}& \begin{array}{c} {\bf 0}_{2\times 1}\\ \hdashline[2pt/2pt] \frac{50}{63} \end{array}
	\end{array} \right].
	\end{aligned}
	\end{equation*}
	
	It is observed that the last rows of the matrices ${\bf H}_{1,2}$ and ${\bf H}_{2,1}$ are the same, as expected from part (c) of Proposition \ref{p1}.
%

In the next part of this example, we elaborate below on the matrix-vector representation  of  the multiplication of two polynomials $\Xi(x)$ and $ \Psi(x)$ of degree $1$ and $2$, respectively, in that basis.  Taking  $n=1$ and $m=2$, we derive from  \eqref{Htilde} that
		\[
	{\bf \tilde{H}}_{1,0} =
		\left[ \begin {array}{cccc} 1&0&0&0\\
		0&1&0&0 \end {array} \right], \qquad
		{\bf \tilde{H}}_{1,1} =
		\left[ \begin {array}{cccc}0&1&0&0\\
		\frac{1}{2}&\frac{2}{5}&\frac{9}{10}&0 \end {array} \right],\qquad
		{\bf \tilde{H}}_{1,2} =
		\left[ \begin {array}{cccc}0&0&1&0\\
		0&\frac{3}{5}&\frac{16}{35}&\frac{6}{7} \end {array} \right].
		\]
		
		Using  (\ref{eqqt1}),  we have
		\begin{equation*} \mathcal{H}_{1,2} =
		\psi_0 {\bf \tilde{H}}_{1,0} + \psi_1 {\bf \tilde{H}}_{1,1} + \psi_2 {\bf \tilde{H}}_{1,2}  =
		\left[ \begin {array}{cccc} \psi_0&\psi_1&\psi_2&0\\ \frac{1}{2} \psi_1&\psi_0 + \frac{2}{5} \psi_1 + \frac{3}{5} \psi_2 & \frac{9}{10}\psi_1 + \frac{16}{35} \psi_2 & \frac{6}{7} \psi_2 \end {array} \right],
		\end{equation*}
		
		and therefore
		\begin{equation*}
         \begin{split}
		\Xi(x) \Psi(x)=
		\left[ \begin {array}{cc}
		\xi_0& \xi_1   \end {array} \right] \mathcal{H}_{1,2} ~~~~~~~~~~~~~~~~~~~~~~~~~~~~~~~~~~~~~~~~~~~~~~~~~~~~~~~~~~~~~~~~~~~~~~~~~~~~~~~~~~
		\\
		= \left[ \begin {array}{cccc}  \xi_0 \psi_0 + \frac{1}{2} \xi_1 \psi_1&
		 \xi_0 \psi_1 +\xi_1 (\psi_0 + \frac{2}{5} \psi_1 +  \frac{3}{5} \psi_2) &\xi_0 \psi_2 + \xi_1 (\frac{9}{10} \psi_1 + \frac{16}{35} \psi_2)& \frac{6}{7} \xi_1 \psi_2 \end {array} \right].
     \end{split}
		\end{equation*}
	\end{Example}

In general, the total complexity for constructing the operational matrix   $\mathcal{ H}_{n,m}$ in \eqref{eqqt1} and then multiplication of two polynomials of degrees $n$ and $m$ is $\mathcal{O}(nm^2)$. However, it is noteworthy that in practice for a fixed polynomial basis,  we can generate the operational matrices $\mathbf{H}_{n,k}, k=0,\cdots,m,$ only once, and use them several times. As such, they deserve to be saved and managed in a set of professional database.
Also, it should be noted that the mentioned complexity cost is the worst-case  and it can be significantly reduced on some practical cases. For example, when we need to multiply $i \geq n$ different  polynomials of degree $n$ into $j \geq m$ different  polynomials of degree $m$, the amortized cost, \cite{BookAlgorithm},  for constructing $\mathbf{H}_{n,k}, k=0,\cdots,m,$ becomes $\mathcal{O}(m)$ or $\mathcal{O}(n)$. Another example is when we have a fixed polynomial of degree $m$ to be multiplied into $i \geq n$ (resp. $i \geq n+m$)  different polynomials of degree $n$, the amortized cost
  for the sequence of $i$ considered  polynomial multiplications  becomes $\mathcal{O}(m^2)$ (resp. $\mathcal{O}(m)$). However, for  the monomial basis, we have the following  result  which shows that in this case, the matrix  $\mathcal{ H}_{n,m}$ is available for free.

\begin{proposition}
	\label{p11}
For the multiplication of two monomial polynomials, we can construct any ${\bf H}_{n,k}$ without any operational cost. In fact,  at this case  we have
\[
	\mathcal{ H}_{n,m} = \sum_{k=0}^{m} \psi_k {\bf \tilde{H}}_{n,k} =
		\left[ \begin {array}{cccccccc}
 \psi_0&\psi_1&\cdots &\psi_m& & & & \\
 & \psi_0&	\psi_1&\cdots&\psi_m& & & \\
  &  & \psi_0&\psi_1 &\cdots & \psi_m &  &  \\
  &          &        &\ddots  & \ddots& \ddots & \ddots & \\
  & &        &        &\psi_0  & \psi_1 &\cdots & \psi_m
  \end {array} \right].
		\]
\end{proposition}

Using  the operational savings stated in Proposition \ref{p11}, for polynomials given in the monomial basis, the only part of (\ref{eqqt1b}) which requires computations is the part
${\bf \xi}^{(n)}\mathcal{H}_{m,n}$.
As such, when $m \leq n$,
$\mathcal{O}(nm)$  operations are required for the multiplication of two polynomials.
\section{Direct Multiplication of Polynomials in Non-Degree-Graded Bases}
This section is on finding the structure of the operational matrices ${\tilde{{\bf H}}}_{n,k}$ for a non-negative integer $k$, where $0\leq k\leq m$, for two major non-degree-graded polynomial bases, i.e.  the Bernstein and the Lagrange polynomial basis.

Based on  Theorem \ref{th1}, these matrices  can be applied to accomplish the task of polynomial multiplications in these two non-degree graded bases.  However,  introducing the ``lifting matrices'' for these two bases, we propose  methods for intra-basis polynomial multiplications  through those matrices. In fact, using lifting matrices is more practical than using the operational matrices  ${\tilde{{\bf H}}}_{n,k}$ for performing intra-basis polynomial multiplications in the Bernstein and Lagrange polynomial bases.

\subsection{Lifting Process}

An important property of non-degree-graded polynomial bases is that one can write a given polynomial in a non-degree-graded polynomial basis of degree $n$ as a polynomial in the same basis of degree $m (> n)$. This seems trivial in degree-graded polynomial bases (that consists of simply adding higher degree basis elements with zero coefficients), but as is  shown in this section, it is quite essential in finding formulas for direct multiplication of polynomials in non-degree-graded polynomial bases. We refer to this property as ``lifting''.


First, according to Lemma \ref{bern1lem}, we can define an $(n+1)\times (n+2)$ ``lifting matrix'' ${\bf T}_{n, n+1}$  of  the Bernstein polynomial basis as
\begin{equation}\label{firstliftbern}
	{\bf T}_{n, n+1}[i,j]=\left\{
	\begin{array}{ll}
		\frac{n+2-i}{n+1}, & i=j\\
		\frac{i}{n+1}, & i= j-1\\
		0, & \text{elsewhere},
	\end{array}\right.
\end{equation}
for $i= 1, \cdots, n+1$ and $j=1, \cdots, n+2$, such that
\begin{equation}\label{matbern}
	{\bf b}_n(x)= {\bf T}_{n, n+1}{\bf b}_{n+ 1}(x).
\end{equation}

It is clear from~\eqref{matbern} that if we want to write a polynomial, $P(x)= {\bf c}^T{\bf b}_n(x)$, given in the Bernstein polynomial basis of degree $n$ in terms of a polynomial in the same basis of degree $m (> n)$, we have $P(x)= {\bf d}^T{\bf b}_m(x)$, where ${\bf d}$ is a vector of size $(m+1)$ given by ${\bf d}^T= {\bf T}_{n, m}{\bf c}^T$ with
\begin{equation}\label{omega1}
	{\bf T}_{n, m}= {\bf T}_{n, n+1}{\bf T}_{n+1, n+2}\cdots{\bf T}_{m-2, m-1}{\bf T}_{m-1, m},
\end{equation}
and ${\bf T}_{n, m}$ is the lifting matrix of size $(n+1)\times(m+1)$.
\\

Using the above results, we can follow an argument to find  the  structure of the lifting  matrix of the Bernstein polynomial basis:
\begin{lemma}\label{liftbernlem}
	The lifting matrix ${\bf T}_{n, m}$ of the Bernstein polynomial basis is an upper triangular matrix with the entries:
	\begin{equation}\label{generalbernlift}
		{\bf T}_{n, m}[i,j]=\left\{
		\begin{array}{ll}
			\frac{\binom{n}{i-1}\binom{m-n}{j-i}}{\binom{m}{j-1}}, & i\leq j\\
			0, & {\rm elsewhere}.
		\end{array}\right.
	\end{equation}
\end{lemma}
\begin{proof}
	For $m= n+1$, using~\eqref{firstliftbern}, we can easily write
	$${\bf T}_{n, n+1}[i,j]=\left\{
	\begin{array}{ll}
	\frac{\binom{n}{i-1}\binom{1}{j-i}}{\binom{n+1}{j-1}}, & i\leq j\\
	0, & \text{elsewhere},
	\end{array}\right.$$ therefore,~\eqref{generalbernlift} holds for $m=n+1$. Note that in this case, the entries are nonzero only when $i=j$ or $i=j-1$. Now, we proceed by induction. We assume the validity of the lemma for $m= k (> n)$ and transit to  $m= k +1$.  A little computation including  matrix multiplications shows that
	$${\bf T}_{k, k+1}[i,j]={\bf T}_{n, k}~ {\bf T}_{k, k+1}=\left\{
	\begin{array}{ll}
	\frac{\binom{n}{i-1}\binom{k-n+1}{j-i}}{\binom{k+1}{j-1}}, & i\leq j\\
	0, & \text{elsewhere},
	\end{array}\right.$$
	i.e., the formula holds for $k+1$ and the proof is complete.
\end{proof}

\begin{Corollary}
	The entries of ${\bf T}_{n, m}$ satisfy the following relation for $i=1, \cdots, n+1$ and $j= 1,\cdots, m+1$:
	\begin{equation}\label{liftingbernmatrixrelation}
		{\bf T}_{n, m}[i, j]= {\bf T}_{n, m}[n-i+2, m-j+2].
	\end{equation}
\end{Corollary}
We now look at the lifting matrix for the Lagrange polynomial basis. Due to the important property of the Lagrange polynomials given in  Lemma \ref{lag1lem}, the entries of an $n\times(n+ 1)$ lifting matrix ${\bf R}_{n-1, n}$  can be defined consequently:
\begin{equation}\label{firstliftlag}
	{\bf R}_{n-1, n}[i,j]=\left\{
	\begin{array}{ll}
		1, & i=j\ne k\\
		-\frac{w_{n,i-1}}{w_{n,k}}, & j= k+1\\
		0, & \text{elsewhere},
	\end{array}\right.
\end{equation}
for $i= 1, \cdots, n$, and $j=1,  \cdots, n+1$, such that
\begin{equation}\label{matlag}
	{\bf L}_{n-1}(x)= {\bf R}_{n-1, n}{\bf L}_{n}(x).
\end{equation}

\begin{Remark}
	Without loss of generality and to make things easier to implement and follow, from this point on, we assume that the added (or eliminated) node is the last one (i.e., $k= n$). Obviously, we can always rearrange the nodes as we wish to make that happen.
\end{Remark}

Now in view of \eqref{matlag},  if we want to write a polynomial $P(x)= {\bf p}^T{\bf L}_n(x)$, given in the Lagrange polynomial basis of degree $n$ as a polynomial in the Lagrange polynomial basis of degree $m (> n)$, we have $P(x)= {\bf q}^T{\bf L}_m(x)$, where ${\bf q}$ is a vector of size $(m+1)$ given by ${\bf q}^T= {\bf R}_{n, m}{\bf c}^T$ with
\begin{equation*}\label{omega2}
	{\bf R}_{n, m}= {\bf R}_{n, n+1}{\bf R}_{n+1, n+2}\cdots{\bf R}_{m-2, m-1}{\bf R}_{m-1, m},
\end{equation*}
and ${\bf R}_{n, m}$ is the lifting matrix of size $(n+1)\times(m+1)$. Similarly, the lifting matrix, ${\bf R}_{n, m}$,  can be constructed as follows:
\begin{lemma}\label{liftlaglem}
	The lifting matrix ${\bf R}_{n, m}$ of the  Lagrange polynomial basis is given by
	\begin{equation}\label{generallaglift}
		{\bf R}_{n, m}=\left[ \begin{array}{c; {2pt/2pt}cc}
			
			{\bf I}_{n+1}& {\bf K}
		\end{array} \right],
	\end{equation}
	where ${\bf I}_{n+1}$ is the identity matrix of size $n+1$ and ${\bf K}$ is an $(n+1)\times (m-n)$ matrix as:
	\begin{equation*}\label{mainlaglift}
		\displaystyle{{\bf K}[i, j]= -{w_{j+n, i-1}\over w_{j+ n, j+ n}}- {1\over w_{j+ n, j+ n}}\sum_{r=1}^{j-1}{w_{j+n, j+n-r}{\bf K}[i, j-r]}},
	\end{equation*}
	for $i= 1, \cdots, n+1$ and $j= 1, \cdots, m-n$.
\end{lemma}
\begin{proof}
	The proof is by induction and the process is quite similar to the proof of Lemma~\ref{liftbernlem}.
\end{proof}

We are now ready to  derive the formulas for direct multiplications of the mentioned non-degree-graded polynomial bases:
\subsection{Bernstein  Basis Multiplication}\label{bernmultp}
Let ${\bf b}_{k, m}(x)$  and  ${\bf b}_{j, n}(x)$ be  the $k$-th and  $j$-th basis elements of the Bernstein polynomial bases of degrees $m$ and $n$, respectively.  Then it is fairly straightforward to observe that
\begin{equation*}
	\displaystyle{{\bf b}_{k, m}(x){\bf b}_{j, n}(x)= {\binom{m}{k}\binom{n}{j}\over\binom{m+n}{k+j}}{\bf b}_{k+j, m+n}(x)},
\end{equation*}
where ${\bf b}_{k+j, m+n}(x)$ is the $(k+j)$-th basis element of the Bernstein polynomial basis of degree $m+n$. In other words, one can write $$\displaystyle{{\bf b}_{k, m}(x){\bf b}_{j, n}(x)= {\bf \tilde{h}}^T_{j,k}{\bf b}_{m+n}(x)},$$ where vector ${\bf b}_{m+n}(x)= \mat{cccc} b_{0, n+m}(x)& b_{1, n+m}(x)& \cdots & b_{n+m, n+m}(x)\rix^T$ is the Bernstein polynomial basis of degree $m+n$, and
%
\begin{equation}\label{bernmat2}
	{\bf \tilde{h}}_{j,k}[i]= \left\{
	\begin{array}{ll}
		\frac{\binom{m}{k}\binom{n}{j}}{\binom{m+n}{k+j}}, & i= k+ j+ 1, \\
		0, & \text{elsewhere.}
	\end{array}
	\right.
\end{equation}

We can now extend this result to ${\bf b}_{k, m}(x){\bf b}_{n}(x)$, 
and state the following lemma whose proof is fairly straightforward:

\begin{lemma}\label{bernelementmul2}
	In the Bernstein polynomial basis, we have
	\begin{equation*}
{\bf b}_{k, m}(x){\bf b}_{n}(x)= {\bf {\tilde{H}}}_{n, k}{\bf b}_{m+n}(x), \qquad   k= 0, \cdots, m
\end{equation*}
 where
	\begin{equation*}
		{\bf \tilde{H}}_{n, k}= \left[
		\begin{array}{ccc}
			{\bf {\tilde{h}}}^T_{0, k} \\ \hdashline[2pt/2pt]
			{\bf {\tilde{h}}}^T_{1, k} \\ \hdashline[2pt/2pt]
			\vdots \\ \hdashline[2pt/2pt]
			{\bf \tilde{h}}^T_{n, k}
		\end{array}
		\right]
	\end{equation*}
	and ${\bf \tilde{h}}_{j,k}$ for $j= 0, \cdots, n$ is given by~\eqref{bernmat2}.
\end{lemma}

The above matrix ${\bf {\tilde{H}}}_{n, k}$ can now be used in Theorem~\ref{th1} to find ${\bf \mathcal{H}}_{n,m}$ and from there calculate intra-basis polynomial multiplications in the Bernstein polynomial basis. However, we derive our formula based on the lifting matrix for the Bernstein polynomial basis given by~\eqref{generalbernlift}.

The following lemma is an important result from Lemma~\ref{liftbernlem} that is used in devising a polynomial multiplication formula in the Bernstein polynomial basis:
\begin{lemma}\label{bernelementmul}
	The multiplication of ${\bf b}_{i, m}(x)$ by ${\bf b}_{j, n}(x)$ for $i=0, 1, \cdots, m$ and $j= 0, 1, \cdots, n$ can be written as
	\begin{equation*}\label{bernelementmul}
		\displaystyle{{\bf b}_{i, m}(x){\bf b}_{j, n}(x)= {\bf T}_{m, m+n}[i+1, i+ j+ 1]{\bf b}_{i+j, m+n}(x)},
	\end{equation*}
	where ${\bf T}_{m, m+n}$ is the lifting matrix given by~\eqref{generalbernlift}.
\end{lemma}
\begin{proof}
	Suppose that $$\Xi(x)= \mat{ccccc}\xi_0& \cdots &\xi_m\rix{\bf b}_{m}(x)= {\bf \xi}^{(m)}{\bf b}_{m}(x) ,$$ and $$\Psi(x)= \mat{ccccc}\psi_0& \cdots &\psi_n\rix{\bf b}_{n}(x)= {\bf \psi}^{(n)}{\bf b}_{n}(x),$$ are two polynomials given in the Bernstein polynomial basis of degrees $m$ and $n$ over $[a, b]$, respectively. Using~\eqref{omega1}, the lifting matrix for $\Xi(x)$ is ${\bf T}_{m, m+n}$ and the lifting matrix for $\Psi(x)$ is ${\bf T}_{n, m+n}$.
\end{proof}

We are now ready to state the following Theorem on the multiplication of two polynomials given in the Bernstein polynomial basis using Lemma~\ref{liftbernlem} and ~\eqref{generalbernlift}:
\begin{Theorem}\label{bernmullem}
	For the Bernstein polynomials $\Xi(x)$ and $\Psi(x)$ given above,
	\begin{equation}
		\Xi(x)\Psi(x)= {\bf c}{\bf b}_{m+n}(x),
	\end{equation}
	where ${\bf c}$ is  a size $m+n+1$ row vector, can be found either in the form of
\begin{equation*}
{\bf c}= {\bf \xi}^{(m)}{\bf \Gamma}_{\psi, m+n},
\end{equation*}
with 		\begin{equation}\label{gammapsi}{\bf \Gamma}_{\psi, m+n}[i, j]= \psi_{j-i}{\bf T}_{m, m+n}[i, j], \end{equation}
	
	and  ${\bf T}_{m, m+n}$ is obtained from~\eqref{generalbernlift} where $\psi_{j-i}=0,$ for  $j< i$.

Correspondingly,
$${\bf c}= {\bf \psi}^{(n)}{\bf \Gamma}_{\xi, m+n},$$ where
	\begin{equation}\label{gammaxi}
		{\bf \Gamma}_{\xi, m+n}[i, j]= \xi_{j-i}{\bf T}_{n, m+n}[i, j],
	\end{equation}
	with ${\bf T}_{n, m+n}$ from~\eqref{generalbernlift} and $\xi_{j-i}=0,$ wherever $j< i$.
\end{Theorem}
\vspace{0.1cm}

For instance, we look at the following example:
\begin{Example}
	\rm Let us assume  $P_1(x)= \mat{cccc}\xi_0 &\xi_1 & \xi_2 & \xi_3\rix {\bf b}_3(x)$ and $P_2(x)= \mat{ccc}\psi_0& \psi_1 & \psi_2\rix{\bf b}_2(x)$ are two polynomials given in the Bernstein polynomial basis over a certain interval $[a, b]$ of degrees $3$ and $2$, respectively, we want to find ${\bf c}= \mat{ccc}c_0& .... & c_5\rix$ so that $P_3(x)= P_1(x)P_2(x)= {\bf c}{\bf  b}_5(x)$.
	
	From~\eqref{generalbernlift}, we find ${\bf T}_{3, 5}=  \left[ \begin {array}{cccccc} 1&\frac{2}{5}&\frac{1}{10}&&&\\ &
	\frac{3}{5}&\frac{3}{5}&\frac{3}{10}&&\\ &&\frac{3}{10}&\frac{3}{5}&\frac{3}{5}&
	\\ &&&\frac{1}{10}&\frac{2}{5}&1\end {array} \right]$ and $ {\bf \Gamma}_{\psi, 5}=\left[ \begin {array}{cccccc} \psi_{{0}}&\frac{2}{5}\,\psi_{{1}}&\frac{1}{10}\,\psi_{{2}}&&&
	\\ &\frac{3}{5}\,\psi_{{0}}&\frac{3}{5}\,\psi_{{1}}&\frac{3}{10}\,\psi_{{2}}&&
	\\ &&\frac{3}{10}\,\psi_{{0}}&\frac{3}{5}\,\psi_{{1}}&\frac{3}{5}\,\psi_{{2}}&
	\\ &&&\frac{1}{10}\,\psi_{{0}}&\frac{2}{5}\,\psi_{{1}}&\psi_{{2}}
	\end {array} \right],$  which yields $${\bf c}= \mat{cccc}\xi_0 &\xi_1 & \xi_2 & \xi_3\rix{\bf \Gamma}_{\psi, 5}=$$  \begin{small}$$\left[ \begin {array}{cccccc} \xi_{{0}}\psi_{{0}}&\frac{2}{5}\,\xi_{{0}}\psi_{{1}}+\frac{3}{5}
		\,\xi_{{1}}\psi_{{0}}&\frac{1}{10}\,\xi_{{0}}\psi_{{2}}+\frac{3}{5}\,\xi_{{1}}\psi_{{1}}+\frac{3}{10}\,\xi_{{2}
		}\psi_{{0}}&\frac{3}{10}\,\xi_{{1}}\psi_{{2}}+\frac{3}{5}\,\xi_{{2}}\psi_{{1}}+\frac{1}{10}\,\xi_{{3}}\psi_{{0}}
		&\frac{3}{5}\,\xi_{{2}}\psi_{{2}}+\frac{2}{5}\,\xi_{{3}}\psi_{{1}}&\xi_{{3}}\psi_{{2}}\end {array}
		\right].$$
	\end{small}
	
	Alternatively, $~{\bf T}_{2, 5}=  \left[ \begin {array}{cccccc} 1&\frac{3}{5}&\frac{3}{10}&\frac{1}{10}&&
	\\ &\frac{2}{5}&\frac{3}{5}&\frac{3}{5}&\frac{2}{5}&\\ &&\frac{1}{10}&
	\frac{3}{10}&\frac{3}{5}&1\end {array} \right]$ and $ {\bf \Gamma}_{\xi, 5}=  \left[ \begin {array}{cccccc} \xi_{{0}}&\frac{3}{5}\,\xi_{{1}}&\frac{3}{10}\,\xi_{{2}}&\frac{1}{10}
	\,\xi_{{3}}&&\\ &\frac{2}{5}\,\xi_{{0}}&\frac{3}{5}\,\xi_{{1}}&\frac{3}{5}\,\xi_{
		{2}}&\frac{2}{5}\,\xi_{{3}}&\\&& \frac{1}{10}\,\xi_{{0}}&\frac{3}{10}\,\xi_{{1}
	}&\frac{3}{5}\,\xi_{{2}}&\xi_{{3}}\end {array} \right]
	,$  which gives $${\bf c}= \mat{ccc}\psi_0 &\psi_1 & \psi_2\rix{\bf \Gamma}_{\xi, 5}=$$  \begin{small}$$\left[ \begin {array}{cccccc} \xi_{{0}}\psi_{{0}}&\frac{2}{5}\,\xi_{{0}}\psi_{{1}}+\frac{3}{5}
		\,\xi_{{1}}\psi_{{0}}&\frac{1}{10}\,\xi_{{0}}\psi_{{2}}+\frac{3}{5}\,\xi_{{1}}\psi_{{1}}+\frac{3}{10}\,\xi_{{2}
		}\psi_{{0}}&\frac{3}{10}\,\xi_{{1}}\psi_{{2}}+\frac{3}{5}\,\xi_{{2}}\psi_{{1}}+\frac{1}{10}\,\xi_{{3}}\psi_{{0}}
		&\frac{3}{5}\,\xi_{{2}}\psi_{{2}}+\frac{2}{5}\,\xi_{{3}}\psi_{{1}}&\xi_{{3}}\psi_{{2}}\end {array}
		\right].$$\end{small}
	
\end{Example}

Consequently, the integer-powers of a polynomial in the Bernstein polynomial basis can be obtained as follows:

\begin{Corollary}
	If $\Xi(x)= {\bf \xi}^{(n)} {\bf b}_n(x)$, then for a positive integer $p (> 1)$ we have:
	\begin{equation*}\label{bernpower}
		\displaystyle{\Xi^{p}(x)= {\bf \xi}^{(n)} (\prod_{j=2}^{p}{\bf \Gamma}_{\xi, (j\times n)}){\bf b}_{(p \times n)}(x)},
	\end{equation*}
	where each ${\bf \Gamma}_{\xi, (j\times n)}$  for $j= 2, 3, \cdots, p$ can be found through~\eqref{gammapsi}.
\end{Corollary}

\subsection{Lagrange Basis Multiplication}\label{lagmultp}

Let ${\bf L}_{k, m}(x)$  and  ${\bf L}_{j, n}(x)$ be  the $k$-th and  $j$-th basis elements of the Lagrange polynomial bases of degrees $m$ and $n$, respectively defined over $\{(\tau_i, p_i)\}_{i=0}^{m}$ and $\{(\tau_i, p_i)\}_{i=0}^{m}$, respectively. Then it is immediately observed that
\begin{equation*}
	\displaystyle{{\bf L}_{k, m}(x){\bf L}_{j, n}(x)= {\bf \tilde{h}}^T_{j,k}{\bf L}_{m+n}(x)},
\end{equation*}
where vector ${\bf L}_{m+n}(x)= \mat{cccc} L_{0, n+m}(x)& L_{1, n+m}(x)& \cdots & L_{n+m, n+m}(x)\rix^T$ is the Lagrange polynomial basis of degree $m+n$ defined over $\{(\tau_i, p_i)\}_{i=0}^{m+n}$, and
%

\begin{equation}\label{lagmat2}
	{\bf \tilde{h}}_{j,k}[i+1]= \left\{
	\begin{array}{lll}
		\delta_{j, k}, & 0\leq i\leq\min{(m, n)}, \\
		0, & \min{(m, n)}+1 \leq i\leq\max{(m, n)}, \\
		{\bf L}_{k, m}(\tau_i){\bf L}_{j, n}(\tau_i), & \max{(m, n)}+1\leq i\leq m+n,
	\end{array}
	\right.
\end{equation}
where  $\delta_{j, k}$ is the Kronecker delta.

We can now extend this result to ${\bf L}_{k, m}(x){\bf L}_{n}(x)$, where vector ${\bf L}_{n}(x)= \mat{cccc} L_{0, n}(x)& L_{1, n}(x)& \cdots & L_{n, n}(x)\rix^T$ is the Lagrange polynomial basis of degree $n$, and state the following lemma whose proof is fairly straightforward:

\begin{lemma}\label{lagelementmul2}
	In the Lagrange polynomial basis,
\begin{equation*}
{\bf L}_{k, m}(x){\bf L}_{n}(x)= {\bf {\tilde{H}}}_{n, k}{\bf L}_{m+n}(x), \qquad k= 0, 1, \cdots, m
\end{equation*}
 where
	\begin{equation*}
		{\bf \tilde{H}}_{n, k}= \left[
		\begin{array}{ccc}
			{\bf {\tilde{h}}}^T_{0, k} \\ \hdashline[2pt/2pt]
			{\bf {\tilde{h}}}^T_{1, k} \\ \hdashline[2pt/2pt]
			\vdots \\ \hdashline[2pt/2pt]
			{\bf \tilde{h}}^T_{n, k}
		\end{array}
		\right]
	\end{equation*}
	and ${\bf \tilde{h}}_{j,k}$ for $j= 0, 1, \cdots, n$ is given by~\eqref{lagmat2}.
\end{lemma}

Similar to what we have for the Bernstein polynomial basis, here ${\bf {\tilde{H}}}_{n, k}$ can be used in Theorem~\eqref{th1} to find ${\bf \mathcal{H}}_{n,m}$ and from there calculate intra-basis polynomial multiplications in the Lagrange polynomial basis. However, we derive the intra-basis multiplication formula based on the lifting matrix for the Lagrange polynomial basis given by~\eqref{generallaglift}.

Recall that if the Lagrange nodes and values of two polynomials  $P(x)$ and $Q(x)$ given in the Lagrange polynomial basis of the same degree are $\{(\tau_i, p_i)\}_{i=0}^{n}$ and $\{(\tau_i, q_i)\}_{i=0}^{n}$, respectively, then the Lagrange values of their multiplication $P(x)Q(x)$ at the same nodes are $\{(\tau_i, p_iq_i)\}_{i=0}^{n}$. This important property is the key to finding a multiplication formula for polynomials given in Lagrange polynomial basis.

Suppose that $$\Xi(x)= \mat{ccccc}\xi_0&\xi_1& \cdots &\xi_{m-1}&\xi_m\rix {\bf L}_{m}(x)= {\bf \xi}^{(m)}{\bf L}_{m}(x),$$ defined by the set of nodes $\{\tau_i\}_{i=0}^{m}$ and $$\Psi(x)= \mat{ccccc}\psi_0&\psi_1& \cdots &\psi_{n-1}&\psi_n\rix {\bf L}_{n}(x)= {\bf \psi}^{(n)}{\bf L}_{n}(x),$$ defined by the set of nodes $\{\tau_i\}_{i=0}^{n}$ are two polynomials given in the Lagrange polynomial basis of degrees $n$ and $m (\geq n)$, respectively. Adding additional nodes $\{\tau_i\}_{i=m+1}^{m+n}$ and using~\eqref{generallaglift}, we can write the polynomials as
\begin{equation*}
	\Xi(x)= \mat{ccccc}s_0&s_1& \cdots &s_{n+m-1}&s_{n+m}\rix{\bf L}_{m+n}(x)= {\bf s}^{(n+ m)}{\bf L}_{m+n}(x),
\end{equation*}
where  $${\bf s}^{(n+ m)}=  {\bf \xi}^{(m)} {\bf R}_{m, m+n},$$ and
\begin{equation*}
	\Psi(x)= \mat{ccccc}t_0&t_1& \cdots &t_{n+m-1}&t_{n+m}\rix{\bf L}_{m+n}(x)= {\bf t}^{(n+ m)}{\bf L}_{m+n}(x),
\end{equation*}
with $${\bf t}^{(n+ m)}{\bf L}_{m+n}(x)= {\bf \psi}^{(n)}{\bf L}_{m+n}(x) {\bf R}_{n, m+n}.$$
\vspace{0.1cm}
\begin{Theorem}\label{lagmullem}
	For the Lagrange polynomials $\Xi(x)$ and $\Psi(x)$ given above,
	\begin{equation*}
		\Xi(x)\Psi(x)= \mat{ccccc}s_0t_0&s_1t_1& \cdots &s_{n+m-1}t_{n+m-1}&s_{n+m}t_{n+m}\rix{\bf L}_{m+n}(x).
	\end{equation*}
\end{Theorem}

In particular for the integer-powers of a polynomial in the Lagrange polynomial basis, we have
\begin{Corollary}
	If $\Xi(x)=  {\bf \xi}^{(n)} {\bf L}_n(x)$, then $\Xi^p (x)$ for a positive integer $p (> 1)$ and with additional nodes $\{\tau_i\}_{i= n+1}^{(p\times n)}$ is
	
	\begin{equation}\label{lagpower}
		\displaystyle{\Xi^{p}(x)= \mat{ccccc}s_0^{p}&s_1^{p}& \cdots &s_{((p\times n)-1)}^{p} &s_{(p\times n)}^{p}\rix{\bf L}_{(p\times n)}(x)},
	\end{equation}
	where $\mat{ccccc}s_0&s_1& \cdots &s_{((p\times n)-1)}&s_{(p\times n)}\rix= {\bf \xi}^{(n)}{\bf R}_{n, (p\times n)}$.
\end{Corollary}

This is illustrated in the following example:
\begin{Example} \rm If $P(x)= \mat{ccc}p_0 & p_1 & p_2 \rix {\bf L}_2(x)$ is a polynomial given in the Lagrange polynomial basis of degree $2$ at the nodes $\{\tau_k\}_{k=0}^{2}$, our aim is  to find ${\bf q}= \mat{ccccc}q_0& q_1& q_2& q_3& q_4\rix$ so that $P^2(x)= {\bf q}{\bf L}_4(x)$ at the nodes $\{\tau_k\}_{k=0}^{4}$.
	
	We first need to find ${\bf R}_{2, 4}$ using~\eqref{generallaglift}:
	\begin{equation*}
		{\bf R}_{2, 4}=  \left[ \begin {array}{ccccc} 1&0&0&-{\frac {w_{{3,0}}}{w_{{3,3}}}}&{
			\frac {w_{{3,0}}w_{{4,3}}-w_{{4,0}}w_{{3,3}}}{w_{{3,3}}w_{{4,4}}}}
		\\ 0&1&0&-{\frac {w_{{3,1}}}{w_{{3,3}}}}&{\frac {w_{
					{3,1}}w_{{4,3}}-w_{{4,1}}w_{{3,3}}}{w_{{3,3}}w_{{4,4}}}}
		\\ 0&0&1&-{\frac {w_{{3,2}}}{w_{{3,3}}}}&{\frac {w_{
					{3,2}}w_{{4,3}}-w_{{4,2}}w_{{3,3}}}{w_{{3,3}}w_{{4,4}}}}\end {array}
		\right].
	\end{equation*}
	
	Next, we find $$\mat{ccccc}s_0&s_1&s_2&s_3&s_4\rix= \mat{ccc}p_0&p_1&p_2\rix{\bf R}_{2, 4},$$ which yields:
	
	\begin{equation*}
		\left\{
		\begin{array}{ll}
			s_0= p_0, &  \\
			s_1= p_1, &  \\
			s_2= p_2, &  \\
			s_3= -{\frac {p_{{0}}w_{{3,0}}}{w_{{3,3}}}}-{\frac {p_{{1}}w_{{3,1}}}{w_{{3
							,3}}}}-{\frac {p_{{2}}w_{{3,2}}}{w_{{3,3}}}}
			, &  \\
			s_4= p_{{0}} \left( -{\frac {w_{{4,0}}}{w_{{4,4}}}}+{\frac {w_{{3,0}}w_{{4,
							3}}}{w_{{3,3}}w_{{4,4}}}} \right) +p_{{1}} \left( -{\frac {w_{{4,1}}}{
					w_{{4,4}}}}+{\frac {w_{{3,1}}w_{{4,3}}}{w_{{3,3}}w_{{4,4}}}} \right) +
			p_{{2}} \left( -{\frac {w_{{4,2}}}{w_{{4,4}}}}+{\frac {w_{{3,2}}w_{{4,
							3}}}{w_{{3,3}}w_{{4,4}}}} \right). &
		\end{array}
		\right.
	\end{equation*}
	Finally, from~\eqref{lagpower} and for $i=0, 1, 2, 3, 4$, we have $q_i= s_i^2$.
\end{Example}

\section{An Application in Stochastic Galerkin Schemes}
Stochastic finite element method \cite{GhanemSpanos} is a  well-known approach for alleviating data uncertainty through the numerical solution of partial differential equations. The main characteristic of such stochastic Galerkin methods is a variational formulation in which the projection  spaces consist of random fields rather than deterministic functions.

In this section, we  give a result to generate  the stochastic Galerkin matrices. These matrices  arise in the discretization of linear differential equations with the random coefficient  and  possess attractive structural and sparsity properties (see e.g.,  \cite {GhanemSpanos, ernst2010Siam, Ullmann2012} and references therein). Besides being useful in the  analysis of stochastic Galerkin schemes, we hope that our result would be useful in developing   efficient iterative  methods for solving the related linear systems.

According to \cite{Ullmann2012}, for an appropriate $M$ and a finite multi-index set $\mathcal{F} \subseteq \mathbb{N}_0^M$, the discretized  linear system related to  the stochastic Galerkin finite element method has the form ${\bf \hat{A}}u=f$, where
\begin{equation*}
	{\bf \hat{A}}=\sum_{\alpha \in \mathcal{F}} {\bf K}_{\alpha} \otimes {\bf G}_{\alpha},
\end{equation*}
is the sum of Kronecker products of matrices in which ${\bf K}_{\alpha}$ and  ${\bf G}_{\alpha}$ are associated with the deterministic/stochastic function spaces.

For $\alpha=(\alpha_1,\ldots,\alpha_M) \in \mathcal{F}$, the stochastic Galerkin matrix  ${\bf G}_{\alpha}$ is given by ${\bf G}_{\alpha}={\bf U}_{\alpha_M, p_M} \otimes \ldots \otimes {\bf U}_{\alpha_2,p_2} \otimes {\bf U}_{\alpha_1,p_1}$, where
for each $m=1,\ldots,M$, ${\bf U}_{\alpha_m, p_m}$ is a $(p_m+1)\times (p_m+1)$ matrix with entries in the form of
\begin{equation*}
	{\bf U}_{\alpha_m,p_m}[i,j]=\Bigl \langle \psi_{\alpha_m}(\xi_m) \psi_{i}(\xi_m)\psi_{j}(\xi_m) \Bigr\rangle, ~~~~~i,j=0,\ldots,p_m.
\end{equation*}

Here  $\langle~ .~ \rangle$ denotes the expectation with respect to a specified probability density $\rho_m$ and $\{\psi_j(\xi_m)\}$ are univariate orthonormal polynomials with respect to the weight function $\rho_m$ appearing in  polynomial chaos (PC) expansion of the coefficient term as a  random field.

As such, in general we need to compute the typical $(p+1) \times (p+1)$ matrix ${\bf U}_{k,p}$ whose entries are given by
\begin{equation}\label{UEntires}
	{\bf U}_{k,p}[i,j]=\bigl \langle \psi_{k} \psi_{i}\psi_{j} \bigr\rangle,\qquad i,j=0,\ldots,p,
\end{equation}
where $\{\psi_j(y)\}$ are univariate orthonormal polynomials  with respect to a specified weight function $\rho(y)$:
\begin{equation*}\label{orthonormality}
	\bigl \langle  \psi_{i}\psi_{j} \bigr\rangle =\int \psi_{i}(y)\psi_{j}(y) \rho(y) dy=\delta_{ij}, \qquad i,j=0,\ldots,p.
\end{equation*}

The following lemma shows that  the univariate  stochastic Galerkin matrix ${\bf U}_{k,p}$ is a principal submatrix of the corresponding   matrix ${\bf H}_{p,k}$ given in Lemma~\ref{lem1}:

\begin{lemma}\label{stochasticLemma}
	For any $k\geq 1$, the $(p+1) \times (p+1)$ univariate Galerkin matrix  ${\bf U}_{k,p}$ is built up from the first $p+1$ rows and columns of the operational matrix ${\bf H}_{p,k}$, i.e.
	\begin{equation*}
		{\bf U}_{k,p}={\bf H}_{p,k} \left(1:p+1,~1:p+1\right).
	\end{equation*}
\end{lemma}


Lemma \ref{stochasticLemma} enables us to use our matrices ${\bf H}_{p,k}$ to construct the univariate stochastic Galerkin matrices ${\bf U}_{k,p}$ for any type of degree-graded (especially orthogonal) polynomial that may appear in PC expansion of the random coefficient term.
For ease of exposition of  the lemma, we work with  the same conditions as outlined in \cite{ernst2010Siam} to obtain the univariate stochastic  Galerkin matrices,  ${\bf U}_{k,p}$, for two main classes of  orthonormal polynomial bases:

{\bf Case 1.} When the univariate or multivariate random variables have normal distributions, the PC expansion basis consists of Hermite orthogonal polynomials, \cite{GhanemSpanos}. Let $H_n(x)$ be the $n$-th basis element of the  Hermite polynomial basis. It is known that for  these polynomials, the three term recurrence relation is
\begin{equation*}
	xH_{n}(x)=\frac{1}{2}H_{n+1}(x)+nH_{n-1}(x),\qquad \quad n=0,1,...
\end{equation*}
with $H_{-1}(x)=0,~H_0(x)=1$ (see e.g. \cite{Andrews1999Cambridge}).   Let  $\psi_n(x)=\frac{1}{\sqrt{2^nn!}} H_n(\frac{x}{\sqrt{2}})$ represent the normalization of $H_n$ such that $\|\psi_n\|^2=1$ with respect to the standard Gaussian weight function $\rho(x)=\frac{{\rm e}^{-x^2/2}}{\sqrt{2 \pi}}$. Under these conditions, the three-term recurrence relation of the orthonormal polynomials, $\psi_n$,  can be described by
$x \psi_n=\sqrt{n+1} \psi_{n+1}+ \sqrt{n} \psi_{n-1}, \quad n=0,1,\ldots .$

Using the results of Lemma \ref{lem1}  and Proposition \ref{p1} to construct the operational matrix  ${\bf H}_{p,k}$  for the Hermite polynomial basis with $k=3$ and $p=5$, we have
\begin{equation*}
	{\bf H}_{5,3}= \begin{bmatrix}
		& & & 1& & & & & \\
		& & \sqrt{3}& & 2& & & & \\
		& \sqrt{3}& & 3 \sqrt{2}& & \sqrt{10}& & & \\
		1& & 3\sqrt{2}& & 3 \sqrt{6}& & 2\sqrt{5}& &\\
		& 2& & 3 \sqrt{6}& 0& 2\sqrt{30}& & \sqrt{35}& \\
		& & \sqrt{10}& & 2\sqrt{30}& & 15& & 2\sqrt{14}
	\end{bmatrix},
\end{equation*}
and from there the univariate Galerkin matrix, ${\bf U}_{3,5}$,  for the Hermite polynomial basis becomes
\begin{equation*}
	{\bf U}_{3,5}= \begin{bmatrix}
		& & & 1& & \\
		& & \sqrt{3}& & 2& \\
		& \sqrt{3}& & 3 \sqrt{2}& & \sqrt{10}\\
		1& & 3\sqrt{2}& & 3 \sqrt{6}&\\
		& 2& & 3 \sqrt{6}& & 2\sqrt{30}\\
		& & \sqrt{10}& & 2\sqrt{30}&
	\end{bmatrix}.
\end{equation*}


{\bf Case 2.} When the univariate or multivariate random variables have uniform distributions, the PC expansion basis consists of Legendre orthogonal polynomials\cite{GhanemSpanos}.

The operational matrix  ${\bf H}_{p,k}$  for the orthonormal Legendre polynomial basis can be used  for obtaining the stochastic Galerkin matrix $U_{k,p}$. For example with $k=4$ and $p=3$,   we get
%
%
\begin{equation*}
	{\bf U}_{4,3}= \begin{bmatrix}
		&  &  &  \\
		&  &  & \frac{4}{\sqrt{21}}\\
		&  & \frac{6}{7}  &  \\
		&  \frac{4}{\sqrt{21}}&  & \frac{6}{11}
	\end{bmatrix}.
\end{equation*}

Both of the above examples can be confirmed by the  explicit formulas   reported in \cite[Appendix A]{ernst2010Siam}.

\section{Concluding Remarks}

Formulas and techniques for intra-basis polynomial multiplication are given in this work with emphasis on degree-graded and non-degree-graded polynomial bases.
In particular, the important role that the operational matrix ${\bf \tilde{H}}_{n,k}$ plays in the process of intra-basis multiplications of polynomials is highlighted. Note that this work does not exhaustively study the computational complexity and numerical analysis of the presented techniques. One can devise methods for making these polynomial multiplication algorithms faster and more accurate.

It is shown in this work that the matrices appearing in the stochastic Galerkin discretization of linear PDEs with random coefficients are sub-matrices of matrices derived for intra-basis multiplication in degree-graded polynomial bases (i.e., ${\bf H}_{p,k}$). To that end, all one needs is the three-term recurrence coefficients in the polynomial degree-graded basis under consideration. Not only are our results useful in constructing those stochastic Galerkin matrices, but we also hope that due to the recursive structure of ${\bf H}_{p,k}$, our results can be used in devising efficient iterative techniques for solving linear systems occurring from the stochastic Galerkin discretization.
More importantly, we expect these intra-basis techniques to find applications in real-world problems where direct and reliable multiplication and division of functions approximated by special polynomials are of utmost importance.


%

\end{document}